\documentclass[12pt,reqno]{amsart}
\usepackage{amsmath,amssymb,amsthm,bbm,mathtools,calc,verbatim,enumitem,tikz,url,mathrsfs,cite,fullpage,hyperref}

\newtheorem{theorem}{Theorem}[section]
\newtheorem{lemma}[theorem]{Lemma}
\newtheorem{corollary}[theorem]{Corollary} 
\newtheorem{conj}[theorem]{Conjecture}

\newtheorem{claim}[theorem]{Claim}
\theoremstyle{definition}
\newtheorem{defn}[theorem]{Definition}
\newtheorem*{qu*}{Question}
\theoremstyle{remark}

\newcommand\N{\mathbb{N}}

\newcommand\cF{\mathcal{F}}

\renewcommand\Pr{\operatorname{\mathbb{P}}}

\newcommand\eps{\varepsilon}
\renewcommand\leq{\leqslant}
\renewcommand\geq{\geqslant}
\renewcommand\le{\leqslant}
\renewcommand\ge{\geqslant}
\renewcommand\to{\rightarrow}

\renewcommand{\leq}{\leqslant}
\renewcommand{\geq}{\geqslant}
\renewcommand{\le}{\leqslant}
\renewcommand{\ge}{\geqslant}
\renewcommand{\to}{\rightarrow}

\let\eps\varepsilon

\def\Ex{\mathbb{E}}

\def\N{\mathbb{N}}

\def\Pr{\mathbb{P}}

\def\<{\langle}
\def\>{\rangle}
\def\0{\mathbf{0}}
\def\1{\mathbbm{1}}

\begin{document}

\title[Set-colouring Ramsey numbers]{A lower bound for Set-colouring Ramsey numbers}
\author{Lucas Arag\~ao \and Maur\'icio Collares \and Jo\~ao Pedro Marciano \and \\
Ta\'isa Martins \and Robert Morris}

\thanks{This research was supported by:~(M.C.) the Austrian Science Fund (FWF P36131); (T.M.) CNPq (Proc.~406248/2021-4); (R.M.) FAPERJ (Proc.~E-26/200.977/2021) and CNPq (Proc.~303681/2020-9).}

\address{IMPA, Estrada Dona Castorina 110, Jardim Bot\^anico, Rio de Janeiro, RJ, Brazil}
\email{\{joao.marciano,l.aragao,rob\}@impa.br}

\address{Institute of Discrete Mathematics, Graz University of Technology, Steyrergasse 30, 8010 Graz, Austria}
\email{mauricio@collares.org}

\address{Instituto de Matem\'atica, Universidade Federal Fluminense,
	Niter\'oi, Brazil}
\email{tlmartins@id.uff.br}

\begin{abstract}
The set-colouring Ramsey number $R_{r,s}(k)$ is defined to be the minimum $n$ such that if each edge of the complete graph $K_n$ is assigned a set of $s$ colours from $\{1,\ldots,r\}$, then one of the colours contains a monochromatic clique of size $k$. The case $s = 1$ is the usual $r$-colour Ramsey number, and the case $s = r - 1$ was studied by Erd\H{o}s, Hajnal and Rado in 1965, and by Erd\H{o}s and Szemerédi in 1972. 

The first significant results for general $s$ were obtained only recently, by Conlon, Fox, He, Mubayi, Suk and Verstra\"ete, who showed that $R_{r,s}(k) = 2^{\Theta(kr)}$ if $s/r$ is bounded away from $0$ and $1$. In the range $s = r - o(r)$, however, their upper and lower bounds diverge significantly. In this note we introduce a new (random) colouring, and use it to determine $R_{r,s}(k)$ up to polylogarithmic factors in the exponent for essentially all $r$, $s$ and $k$. 
\end{abstract}

\maketitle

\section{Introduction}\label{sec:introduction}

The $r$-colour Ramsey number $R_r(k)$ is defined to be the minimum $n \in \N$ such that every $r$-colouring $\chi \colon E(K_n) \to \{1,\ldots,r\}$ of the edges of the complete graph on $n$ vertices contains a monochromatic clique of size $k$. These numbers (and their extensions to general graphs, hypergraphs, etc.) are among the most important and extensively-studied objects in combinatorics, see for example the beautiful survey article~\cite{CFS}.  

In this paper we will study the following generalisation of the $r$-colour Ramsey numbers. 

\begin{defn}
The set-colouring Ramsey number $R_{r,s}(k)$ is the least $n \in \N$ such that every colouring $\chi \colon E(K_n) \to \binom{[r]}{s}$ contains a monochromatic clique of size $k$, that is, a set $S \subset V(K_n)$ with $|S| = k$ and a colour $i \in [r]$ such that $i \in \chi(e)$ for every $e \in \binom{S}{2}$.
\end{defn}

In other words, we assign a set $\chi(e) \subset [r] = \{1,\ldots,r\}$ of $s$ colours to each edge of the complete graph, and say that a clique is monochromatic if there exists a colour $i \in [r]$ that is assigned to every edge of the clique. 
Note that when $s = 1$ this is simply the usual $r$-colour Ramsey number $R_r(k)$, for which the best known bounds are 
$$2^{\Omega(kr)} \le R_r(k) \le r^{O(kr)}.$$
Both bounds have simple proofs: the upper bound follows from the classical method of Erd\H{o}s and Szekeres~\cite{ESz35}, while the lower bound can be proved using a simple product colouring due to Lefmann~\cite{L}. Determining whether or not $R_r(k) = 2^{\Theta(kr)}$ is a major and longstanding open problem, see for example the recent improvements of the lower bound in~\cite{CF,W,S}. 

The study of set-colouring Ramsey numbers was initiated in the 1960s by Erd\H{o}s, Hajnal and Rado~\cite{EHR}, who conjectured that $R_{r,r-1}(k) \le 2^{\delta(r) k}$ for some function $\delta(r) \to 0$ as $r \to \infty$. This conjecture was proved by Erd\H{o}s and Szemerédi~\cite{ESz} in 1972, who showed that in fact
$$2^{\Omega(k/r)} \le R_{r,r-1}(k) \le r^{O(k/r)}.$$
The lower bound follows from a simple random colouring, while to prove the upper bound Erd\H{o}s and Szemerédi showed that any $2$-colouring in which one of the colours has density at most $1/r$ contains a monochromatic clique of size $c(r) \log n$, where $c(r) = \Omega( r / \log r )$, and then applied this result to the colour that is assigned to most edges.

For more general values of $s$, the first significant progress was made only recently, by Conlon, Fox, He, Mubayi, Suk and Verstra\"ete~\cite{CFHMSV}, who showed that 
$$R_{r,s}(k) = 2^{\Theta(kr)}$$ 
for every function $s = s(r)$ such that $s/r$ is bounded away from $0$ and $1$. More precisely, they used a clever generalisation of the approach of Erd\H{o}s and Szemerédi~\cite{ESz} to prove that 
\begin{equation}\label{CFHMSV:upper}
R_{r,s}(k) \le \exp\bigg( \frac{ck(r-s)^2}{r} \log \frac{r}{\min\{s,r-s\}} \bigg)
\end{equation}
for some absolute constant $c > 0$, and defined a product colouring, which exploited an interesting and surprising connection with error-correcting codes, to show that
\begin{equation}\label{CFHMSV:lower}
R_{r,s}(k) \ge \exp\bigg( \frac{c'k(r-s)^3}{r^2} \bigg)
\end{equation}
for some constant $c' > 0$. They also noted 
that a simple random colouring gives a lower bound of the form $R_{r,s}(k) \ge 2^{\Omega(k(r-s)/r)}$, which is stronger when $r - s \ll \sqrt{r}$. 

While the exponents in~\eqref{CFHMSV:upper} and~\eqref{CFHMSV:lower} differ by only a factor of $\log r$ when $r - s = \Omega(r)$, they diverge much more significantly when $(r - s)/r \to 0$.\footnote{We remark that the range $s = r - o(r)$ was of particular interest to the authors of~\cite{CFHMSV}, who were motivated by an application to hypergraph Ramsey numbers, see~\cite{CFHMSVapp}.} The main result of this paper is the following improved lower bound, which will allow us to determine $R_{r,s}(k)$ up to a poly-logarithmic factor in the exponent for essentially all $r$, $s$ and $k$. 

\begin{theorem}\label{thm:setRamsey}
There exist constants $C > 0$ and $\delta > 0$ such that the following holds. If $r,s \in \N$ with $s \le r - C\log r$, then 
\begin{equation}\label{eq:setRamsey:lower}
R_{r,s}(k) \ge \exp\bigg( \frac{\delta k (r-s)^2}{r} \bigg)
\end{equation}
for every $k \ge (C / \eps) \log r$, where $\eps = (r-s)/r$.
\end{theorem}

Note that the bound~\eqref{eq:setRamsey:lower} matches the upper bound~\eqref{CFHMSV:upper} on $R_{r,s}(k)$, proved in~\cite{CFHMSV}, up to a factor of $O(\log r)$ in the exponent for all $s \le r - C\log r$. When $s \ge r - C\log r$ our method does not provide a construction, but in this case the bounds from~\cite{CFHMSV} only differ by a factor of order $(\log r)^2$ in the exponent, the lower bound coming from a simple random colouring. 

The lower bound on $k$ in Theorem~\ref{thm:setRamsey} is also not far from best possible, since if $k \le 1/\eps$ then the most common colour has density at least $1 - 1/k$, and therefore $R_{r,s}(k) \le k^2$, by Turán's theorem. We remark that the range in which Turán's theorem provides an optimal bound was investigated in detail by Alon, Erd\H{o}s, Gunderson and Molloy~\cite{AEGM} in the case $s = r - 1$. In Section~\ref{simple:sec} we will describe a simpler version of our construction which proves the lower bound
\begin{equation}\label{eq:simple:bound}
R_{r,s}(k) \ge \bigg( \frac{\eps (k-1)}{e} \bigg)^{\eps r/2}
\end{equation}
for all $r$, $s$ and $k$. Note that this bound matches the upper bound~\eqref{CFHMSV:upper} up to a factor of order $(\log r)^2$ when $3/\eps + 1 \le k \le (C / \eps) \log r$. Moreover, as long as $k \ge (1+\delta)/\eps + 1$ for some constant $\delta > 0$, we can use the same construction to prove (with a slightly more careful calculation) a lower bound of the form
\begin{equation}\label{eq:simple:bound2}
R_{r,s}(k) \ge 2^{\Omega(\eps r)}.
\end{equation}
Thus, writing $\tilde{\Theta}$ to hide poly-logarithmic factors of $r$, we obtain the following corollary.

\begin{corollary}\label{cor:general}
Let $r > s \ge 1$ and $\delta > 0$, and set $\eps = (r-s)/r$. We have
$$R_{r,s}(k) = 2^{\tilde{\Theta}(\eps^2 rk)}$$
for every $k \ge (1+\delta)/\eps + 1$. 
\end{corollary}

It would be interesting to determine the behaviour of $R_{r,s}(k)$ in the range $ \eps k = 1+o(1)$, especially in light of the initial progress made in~\cite{AEGM}. 



\section{The construction}\label{sec:construction}

In this section we will define the (random) colouring that we use to prove Theorem~\ref{thm:setRamsey}, and give an outline of the proof that it has the desired properties with high probability. The idea behind our construction, to let each colour be a random copy of some pseudorandom graph, was introduced in the groundbreaking work of Alon and R\"odl~\cite{AR} on multicolour Ramsey numbers, and has been used in several recent papers in the area~\cite{HW,MV,W,S}. However, our approach differs from that used in these previous works in several important ways; in particular, we will not count independent sets, and it will be important that our colour classes are chosen (almost) independently at random. We will discuss the novel aspects of our construction in more detail at the end of this section. 

Our first attempt at defining a colouring $\chi \colon E(K_n) \to \binom{[r]}{s}$ with no monochromatic $K_k$ is as follows: let the edges $e \in E(K_n)$ with $i \in \chi(e)$ be a random blow-up of the random graph $G(m,p)$, where $p = 1 - \Theta(\eps)$, and $m$ is chosen so that $G(m,p)$ is unlikely to contain a copy of $K_k$. The problem with this colouring is that there will be some `bad' edges which receive fewer than $s$ colours,\footnote{Note that for edges that receive more than $s$ colours, we can take an arbitrary subset of size $s$.} and we will therefore need to modify the construction by giving these edges extra colours. The key idea is to only give them the colours for which they are `crossing' edges, that is, not in the same part of the random blow-up for that colour.

We will next define the colouring precisely. Fix a sufficiently small\footnote{In fact taking $\delta = 2^{-5}$ would suffice, but we will not make any attempt to optimise the constants in Theorem~\ref{thm:setRamsey}.} constant $\delta > 0$, and set $C = 1/\delta^3$. Recall that $r - s = \eps r$, and define
$$m = 2^{\delta^2\eps k} \qquad \text{and} \qquad n = 2^{\delta^4 \eps^2 r k}.$$
Note that $\eps \sqrt{m} \ge k$, since $k \ge (C / \eps) \log r$ and $\eps \ge 1/r$, and by our choice of $C$. 

 
Set $p = 1 - 5\delta\eps$, and for each colour $i \in [r]$, let 
\begin{itemize}
\item $H_i$ be an independently chosen copy of the random graph $G(m,p)$, and\smallskip
\item $\phi_i \colon [n] \to [m]$ be an independently and uniformly chosen random function.
\end{itemize} 
Now define $G_i$ to be the (random) graph with vertex set $[n]$ and edge set
$$E(G_i) = \big\{ uv : \{\phi_i(u), \phi_i(v)\} \in E(H_i) \big\},$$
that is, a random blow-up of $H_i$, with parts given by $\phi_i$. Define a colouring $\chi'$ of $K_n$ by
$$\chi'(e)= \big\{ i \in [r] : e \in E(G_i) \big\},$$  
and define the set of \emph{bad} edges to be
\begin{equation}\label{def:B}
B = \big\{e \in E(K_n) : |\chi'(e)| < s \big\}.
\end{equation}
We will also say that an edge $e = uv \in E(K_n)$ is \emph{$i$-crossing} if $\phi_i(u) \ne \phi_i(v)$, and define
$$\kappa(e) = \big\{ i \in [r] : \text{$e$ is $i$-crossing} \big\}.$$
We can now define the colouring that we will use to prove Theorem~\ref{thm:setRamsey}. 

\begin{defn}\label{def:colouring}
For each $e \in E(K_n)$, we define the set of colours $\chi(e) \subset [r]$ by
  \[
    \chi(e) \,=\, \left\{
    \begin{array} {c@{\quad \textup{if} \quad}l}
      \chi'(e) & e \not\in B, \\[+1ex]
      \kappa(e) & e \in B.
    \end{array}\right.
  \]
\end{defn}

Our task is to show that with high probability $|\chi(e)| \ge s$ for every $e \in E(K_n)$, and moreover that $\chi$ contains no monochromatic copy of $K_k$. Before giving the details, let us briefly outline how we will go about proving these two properties. 

The first property, that $|\chi(e)| \ge s$ for every $e \in E(K_n)$, is a relatively straightforward consequence of the definition of $\chi$ and our choice of $m$. Indeed, by Definition~\ref{def:colouring} and~\eqref{def:B}, it will suffice to show that $|\kappa(e)| \ge s$ for every $e \in E(K_n)$, and this follows from a simple first moment calculation, using the fact that $n = m^{\delta^2 \eps r}$.

Proving that with high probability $\chi$ contains no monochromatic copy of $K_k$ is more difficult, and will be the main task of Section~\ref{proof:sec}. Cliques with few bad edges are easily dealt with using Chernoff's inequality, so let us focus here on cliques with many bad edges, where `many' means at least $t = \delta \eps k^2$. The difficulty in this case is that the events $\{e \in B\}$ and $\{f \in B\}$ are correlated, since the endpoints may be in the same part in some of the random partitions. In particular, if $\phi_i(u) = \phi_i(v)$ for many colours $i \in [r]$ and many pairs $\{u,v\}$ of high-degree vertices of $F$, the graph of bad edges in our monochromatic $k$-clique, then Chernoff's inequality will not provide strong enough bounds on our large deviation events. 

To deal with this issue, we will use the randomness of the partitions $\phi_1,\ldots,\phi_r$ to show (see Lemma~\ref{lem:XF:is:small}) that there is not too much `clustering' of the vertices of \emph{any} graph $F \subset K_n$ with $k$ vertices and $t$ edges. To do so, we will not be able to use a simple union bound over all graphs, since there will be too many choices for the low-degree vertices of $F$; instead, we will need to find a suitable `bottleneck event' for each $F$, and apply the union bound to these events. Roughly speaking, we will find an initial segment $A$ of the vertices of $F$, ordered according to their degrees, with the following property: there are $\delta \eps r |A|$ pairs $v \in A$ and $i \in [r]$ such that there exists an `earlier' vertex $u \in A$ with $\phi_i(u) = \phi_i(v)$. We will then sum over the choices of $A$, using the fact that for each such pair $v \in A$ and $i \in [r]$, this event 
(conditioned on the choices of $\phi_i(u)$ for earlier $u$) has probability at most $k/m$. 

On the other hand, when there is not too much clustering of the high-degree vertices of $F$ in the random partitions $\phi_1,\ldots,\phi_r$, we will use the randomness in the choice of $H_1,\ldots,H_r$ to bound the probability that there are more bad edges than expected. More precisely, we will choose one vertex from each cluster in each colour, and apply Chernoff's inequality, noting that the edges between these vertices are independent. 

The proof outlined above diverges from the method of Alon and R\"odl~\cite{AR} (and the more recent constructions of~\cite{HW,MV,W,S}) in several important ways. Of these new ideas, we note in particular that our choice of $\kappa(e)$ for the bad edges is somewhat subtle (and perhaps surprising), since $\kappa(e) = [r]$ for almost all edges $e \in E(K_n)$, which seems very wasteful. The point is that there are very few bad edges, and by only including crossing colours in $\chi(e)$ we ensure that there is not too much dependence between the colours of different edges. It is also important that the random graphs $H_1,\ldots,H_r$ are chosen independently; if we used the same random graph for each colour, then we would not be able to prove a sufficiently strong bound on the probability that a clique has too many bad edges.   

\section{The proof}\label{proof:sec}

We begin by observing that with high probability $|\chi(e)| \ge s$ for every $e \in E(K_n)$.

\begin{lemma}\label{lem:chi:always:big}
With high probability, $|\chi(e)| \geq s$ for every $e \in E(K_n)$.
\end{lemma}

\begin{proof}
Note first that, by the definition of $B$, if $e$ is not a bad edge then $\chi(e) = \chi'(e)$ and $|\chi'(e)| \ge s$. It will therefore suffice to show that 
\begin{equation}\label{eq:chi:always:big}
\Pr\big( |\kappa(e)| < s \big) \le \frac{1}{n^3}
\end{equation}
for each edge $e \in E(K_n)$. To do so, note that for each $i \in [r]$ we have
$$\Pr\big( i \not\in \kappa(e) \big) = \frac{1}{m},$$
all independently, by the definition of the functions $\phi_i$. Recalling that $r-s = \eps r$ and that $\eps m \ge \sqrt{m} = 2^{\delta^2 \eps k/2}$, since $k \geq (C / \eps) \log r$ and $C = 1/\delta^3$, it follows that
$$\Pr\big( |\kappa(e)| < s \big) \le \binom{r}{\eps r} \bigg( \frac{1}{m} \bigg)^{\eps r} \le \bigg( \frac{e}{\eps m} \bigg)^{\eps r} \le \, 2^{-\delta^3 \eps^2 rk},$$
Since $n = 2^{\delta^4 \eps^2 r k}$, we obtain~\eqref{eq:chi:always:big}, and so the lemma follows by Markov's inequality. 
  \end{proof}

We are left with the (significantly more challenging) task of showing that, with high probability, $\chi$ contains no monochromatic copy of $K_k$. We will first deal with the (easier) case in which the clique has few bad edges. Recall that $t = \delta \eps k^2$. 

  
\begin{lemma}\label{lemma:few-bad}
With high probability, the colouring $\chi$ contains no monochromatic $k$-clique with at most $t$ bad edges.
\end{lemma}

\begin{proof}
Suppose $\chi$ contains a monochromatic clique $S = \{v_1, \ldots, v_k\}$ of colour $i \in [r]$ such that at most $t$ of the edges $e \in {S \choose 2}$ are bad. For each $j \in [k]$, let $w_j = \phi_i(v_j) \in V(H_i)$, and observe that the set $W = \{w_1, \ldots, w_k\}$ has size $k$, since by Definition~\ref{def:colouring}, and noting that $\chi'(e) \subset \kappa(e)$, every edge $e \in E(K_n)$ such that $i \in \chi(e)$ is $i$-crossing. 

Now, if $e = v_jv_\ell \in {S \choose 2}$ is not a bad edge, then $i \in \chi(e) = \chi'(e)$, and hence $w_j w_\ell \in E(H_i)$. Since there are at most $t$ bad edges in ${S \choose 2}$, it follows that 
$$e\big( H_i[W] \big) \ge \binom{k}{2} - t \, > \, p {k \choose 2} + \delta \eps k^2,$$
since $p = 1 - 5\delta \eps$ and $t = \delta \eps k^2$. Since $H_i[W] \sim G(k,p)$, it follows from Chernoff's inequality that this event has probability at most $e^{-\delta^2\eps k^2}$. Therefore, taking a union bound over colours $i \in [r]$ and sets $W \subset V(H_i)$ of size $k$, the probability that $\chi$ contains a monochromatic clique with at most $t$ bad edges is at most
\begin{equation}\label{eq:expected:manyedges:H}
r \binom{m}{k} e^{-\delta^2\eps k^2} \le r \big( 2^{\delta^2 \eps k} \cdot e^{-\delta^2 \eps k} \big)^k.
\end{equation}
Since $\delta^3 \eps k \ge \log r$, the right-hand side of~\eqref{eq:expected:manyedges:H} tends to zero as $k \to \infty$, as required.
\end{proof}

It remains to prove the following lemma, which is not quite so straightforward.

\begin{lemma}\label{lemma:many-bad}
With high probability, every $k$-clique contains at most $t$ bad edges.
\end{lemma}

In other words, our task is to show that, with high probability, there does not exist a graph $F$ in the family
$$\mathcal{F} = \big\{ F \subset K_n \, : \, v(F) = k \,\text{ and }\, e(F) = t \big\}$$
such that $F \subset B$.\footnote{Here, and below, we abuse notation slightly by treating the set of bad edges $B$ as a graph.} We will not be able to prove this using a simple 1st moment argument, summing over all graphs $F \in \cF$, since the probability of the event $\{ F \subset B \}$ is not always sufficiently small. Instead, we will need to identify a `bottleneck event' for each $F \in \cF$. 

To do so, let us first choose an ordering $\prec_F$ on the vertices of $F$ satisfying 
$$u \prec_F v \qquad \Rightarrow \qquad d_F(u) \ge d_F(v).$$
In other words, we order the vertices according to their degrees in $F$, breaking ties arbitrarily. Now, define
$$Q_i(F) = \big\{ v \in V(F) \,:\, \exists \; u \in V(F) \,\text{ with }\, u \prec_F v \,\text{ such that }\, \phi_i(u) = \phi_i(v)\big\}$$  
to be the set of vertices which share a part of $\phi_i$ with another vertex of $F$ that comes earlier in the order $\prec_F$. We remark that if $u \prec_F v$, then $u \ne v$. 

We will bound the probability in two different ways, depending on the size of 
$$X_F = \sum_{i = 1}^r \sum_{v \in Q_i(F)} d_F(v).$$
When $X_F$ is large, we will find an initial segment $A$ of the order $\prec_F$ such that there are at least $\delta \eps r |A|$ pairs $i \in [r]$ and $v \in Q_i(F) \cap A$. First, however, we will deal with the case in which $X_F$ is small, where we can use a simple union bound.


\begin{lemma}\label{lem:small:XF}
With high probability, there does not exist $F \in \mathcal{F}$ with 
$$X_F \le \frac{\eps r t}{2} \qquad \text{and} \qquad F \subset B.$$
\end{lemma}

\begin{proof} 
We first reveal the random functions $\phi_1,\ldots,\phi_r$, and therefore the sets $Q_i(F)$ (and hence also the random variable $X_F$) for each $F \in \cF$. To prove the lemma we will only need to use the randomness in the choice of $H_1,\ldots,H_r$. More precisely, we will consider the set 
$$Y = \big\{ (uv,i) \in E(F) \times [r] : u,v \not\in Q_i(F) \big\}$$
of pairs $(e,i) \in E(F) \times [r]$ such that neither endpoint of $e$ is contained in $Q_i(F)$, and
$$Z = \sum_{(e,i) \in Y} \1\big[ e \not\in E(G_i) \big],$$
the number of such pairs for which $i \not\in \chi'(e)$. 
Note that $|Y| \le rt$, and that
$$Z \sim \textup{Bin}(|Y|,1-p),$$ 
since the events $\{e \in E(G_i) \}$ for $(e,i) \in Y$ are independent, and correspond to the appearance of certain edges in the graphs $H_1,\ldots,H_r$. Indeed, for each $i \in [r]$, the graph $\{ e : (e,i) \in Y \}$ is contained in a clique with at most one vertex in each part of $\phi_i$. 

Thus, given $|Y|$ (which is determined by $\phi_1,\ldots,\phi_r$) the random variable $Z$ is a binomial random variable with expectation
$$\Ex\big[ Z \,\big|\, |Y| \big] \, \le \, (1-p)rt \, = \, 5\delta\eps rt.$$
Now, note that if $F \subset B$, then for each edge $e \in E(F)$, there are at least $\eps r$ colours $i \in [r]$ such that $e \not\in E(G_i)$. Thus
$$\sum_{i=1}^r \sum_{e \in E(F)} \1\big[ e \not\in E(G_i) \big] \ge \eps r t,$$
and moreover, if $X_F \le \eps r t / 2$, then 
$$Z \ge \frac{\eps rt}{2},$$
since for each vertex $v \in Q_i(F)$ we remove at most $d_F(v)$ edges from $Y$. By Chernoff's inequality, it follows that for a fixed $F \in \cF$ we have 
$$\Pr\Big( X_F \le \eps r t / 2 \,\text{ and }\, F \subset B \Big) \le e^{-\delta\eps rt}.$$
Therefore, taking a union bound over $F \in \cF$, and recalling that $t = \delta\eps k^2$, it follows that the probability that there exists $F \in \mathcal{F}$ with $X_F \le \eps r t / 2$ and $F \subset B$ is at most
$$\binom{n}{k}\binom{\binom{k}{2}}{t} e^{-\delta \eps r t} \leq \left(\frac{en}{k}\left(\frac{e}{\delta \eps}\right)^{\delta \eps k} e^{-\delta^2 \eps^2 r k}\right)^k \to 0,$$
as claimed, where in the final step we used our choice of $n = 2^{\delta^4 \eps^2 r k}$, the bound 
$$\eps = \frac{r-s}{r} \ge \frac{C\log r}{r},$$
which holds by our assumption that $s \le r - C\log r$, and our choice of $C = 1/\delta^3$. 
\end{proof}

Finally, we will use the randomness in $\phi_1,\ldots,\phi_r$ to show that $X_F$ is always small. 

\begin{lemma}\label{lem:XF:is:small}
With high probability, 
$$X_F \le \frac{\eps r t}{2}$$
for every $F \in \mathcal{F}$. 
\end{lemma}

\begin{proof}
For each graph $F \in \cF$, and each $j \in \big\{ 1,\ldots,\lceil \log_2 k \rceil \big\}$, define a set
$$A_j(F) = \Big\{ v \in V(F) : 2^{-j} k \le d_F(v) < 2^{-j+1} k \Big\}$$
and a random variable 
$$s_j(F) = \sum_{i=1}^r |A_j(F) \cap Q_i(F)|.$$
Note that the random functions $\phi_1,\ldots,\phi_r$ determine $Q_1(F),\ldots,Q_r(F)$, and hence $s_j(F)$. The key step is the following claim, which provides us with our bottleneck event.  

\begin{claim}
If $X_F \ge \eps r t/2$, then there exists $\ell \in \big\{ 1,\ldots,\lceil \log_2 k \rceil \big\}$ such that
\begin{equation}\label{eq:initial:segment}
s_\ell(F) > \, \delta \eps r \, \sum_{j = 1}^\ell |A_j(F)|.
\end{equation}
\end{claim}

\begin{proof}[Proof of claim]
Observe that
$$X_F = \sum_{i = 1}^r \sum_{v \in Q_i(F)} d_F(v) \le \sum_{i = 1}^r \sum_{j = 1}^{\lceil \log_2 k \rceil} \frac{k}{2^{j-1}} \cdot |A_j(F) \cap Q_i(F)| = \sum_{j = 1}^{\lceil \log_2 k \rceil} \frac{k}{2^{j-1}} \cdot s_j(F),$$ 
and therefore if $X_F \geq \eps r t/2$, then
\begin{equation}\label{eq:apply-hyp-HF}
\sum_{j = 1}^{\lceil \log_2 k \rceil} \frac{s_j(F)}{2^j} \ge \frac{\eps rt}{4k}.
\end{equation}


\noindent Note also that
$$\sum_{j = 1}^{\lceil \log_2 k \rceil} \frac{|A_j(F)|}{2^j} \, \le \, \frac{1}{k} \sum_{v \in V(F)} d_F(v) \, = \, \frac{2t}{k}.$$
Thus, if~\eqref{eq:initial:segment} fails to hold for every $\ell \in \big\{ 1,\ldots,\lceil \log_2 k \rceil \big\}$ then, by~\eqref{eq:apply-hyp-HF}, we have 
\begin{align*}
\frac{t}{4\delta k} & \, \le \, \frac{1}{\delta \eps r} \sum_{\ell = 1}^{\lceil \log_2 k \rceil} \frac{s_\ell(F)}{2^\ell} \, \le \, \sum_{\ell = 1}^{\lceil \log_2 k \rceil} \frac{1}{2^\ell}  \sum_{j = 1}^\ell |A_j(F)| \\
& \hspace{2cm} \, = \, \sum_{j = 1}^{\lceil \log_2 k \rceil} |A_j(F)| \sum_{\ell = j}^{\lceil \log_2 k \rceil} \frac{1}{2^\ell} \, \le \, \sum_{j = 1}^{\lceil \log_2 k \rceil} \frac{|A_j(F)|}{2^{j-1}} \, \le \, \frac{4t}{k}.
\end{align*}
Since $\delta < 2^{-4}$, this is a contradiction, and so the claim follows.
\end{proof}

Fix $\ell \in \big\{ 1,\ldots,\lceil \log_2 k \rceil \big\}$ such that ~\eqref{eq:initial:segment} holds, and set
$$A := \bigcup _{j = 1}^\ell A_j(F) \qquad \text{and} \qquad a := |A|.$$
Now, if we reveal $\phi_i$ for the vertices of $F$ one vertex at a time using the order $\prec_F$, then for each vertex $v \in Q_i(F)$ we must choose $\phi_i(v)$ to be one of the (at most $k$) previously selected elements of $[m]$. The expected number of sets $A$ such that~\eqref{eq:initial:segment} holds is thus at most
$$\sum_{a = 1}^k n^a {ar \choose \delta\eps ar} \bigg( \frac{k}{m} \bigg)^{\delta\eps ar} \le \, \sum_{a = 1}^k \bigg( n \cdot \bigg( \frac{e}{\delta\eps} \cdot \frac{k}{m} \bigg)^{\delta\eps r} \bigg)^a \to 0$$
as $k \to \infty$, as required, since $n = 2^{\delta^4 \eps^2 r k}$ and $\eps m / k \ge \sqrt{m} = 2^{\delta^2\eps k/2}$. 
\end{proof}

Combining Lemmas~\ref{lem:small:XF} and~\ref{lem:XF:is:small}, we obtain Lemma~\ref{lemma:many-bad}.

\begin{proof}[Proof of Lemma~\ref{lemma:many-bad}]
By Lemma~\ref{lem:XF:is:small}, with high probability we have $X_F \le \eps r t / 2$ for every $F \in \cF$. By Lemma~\ref{lem:small:XF}, it follows that with high probability $F \not\subset B$ for every $F \in \cF$. Therefore, with high probability every $k$-clique contains at most $t$ bad edges, as claimed.
\end{proof}

We can now easily put together the pieces and prove our main theorem. 

\begin{proof}[Proof of Theorem~\ref{thm:setRamsey}]
With high probability, the random colouring $\chi$ satisfies:
\begin{itemize}
\item $|\chi(e)| \ge s$ for every $e \in E(K_n)$, by Lemma~\ref{lem:chi:always:big};\smallskip
\item $\chi$ contains no monochromatic $K_k$ with at most $t$ bad edges, by Lemma~\ref{lemma:few-bad};\smallskip
\item every $k$-clique contains at most $t$ bad edges, by Lemma~\ref{lemma:many-bad}.
\end{itemize}
Thus $R_{r,s}(k) > n = 2^{\delta^4 \eps^2 rk}$, as required.
\end{proof}

\section{A simpler construction for small $k$}\label{simple:sec}

To conclude, we will prove the alternative lower bounds~\eqref{eq:simple:bound} and~\eqref{eq:simple:bound2}, promised in the introduction, which hold for smaller values of $k$. Together with Theorem~\ref{thm:setRamsey} and the results of~\cite{CFHMSV}, these bounds will allow us to prove Corollary~\ref{cor:general}.

\begin{theorem}\label{thm:simple:bound}
Let $r > s \ge 1$, and set $\eps = (r-s)/r$. We have
$$R_{r,s}(k) \ge \bigg( \frac{\eps (k-1)}{e} \bigg)^{\eps r/2}$$
for every $k \in \N$. 
\end{theorem}

Note that Theorem~\ref{thm:simple:bound} only gives a non-trivial bound when $\eps (k-1) > e$. However, the same construction (together with a slightly more careful calculation) in fact gives an exponential lower bound almost all the way down to the Turán range $k \le 1/\eps$. 

\begin{theorem}\label{thm:simple:bound2}
Let $r > s \ge 1$ and $\delta > 0$, and set $\eps = (r-s)/r$. We have
$$R_{r,s}(k) \ge \exp\big( c(\delta) \eps r \big)$$
for some $c(\delta) > 0$ and every $k \ge (1 + \delta)/\eps + 1$.
\end{theorem}

The construction is a simpler version of the one we used to prove Theorem~\ref{thm:setRamsey}: instead of taking blow-ups of a random graph for our colours, we take blow-ups of the complete graph with $k-1$ vertices (that is, we take complete $(k-1)$-partite graphs). In particular, this means that we do not need to worry about crossing edges. 

To define the colouring precisely, let $\phi_1,\ldots,\phi_r$ be independent uniformly chosen functions from $[n]$ to $[k-1]$, and for each $i \in [r]$ define $G_i$ to be the (random) graph with vertex set $[n]$ and edge set
$$E(G_i) = \big\{ uv : \phi_i(u) \ne \phi_i(v) \big\},$$
that is, a random complete $(k-1)$-partite graph, with parts given by $\phi_i$. Define a colouring $\chi$ of the edges of $K_n$ by
$$\chi(e)= \big\{ i \in [r] : e \in E(G_i) \big\}.$$  
Since each colour class is $(k-1)$-partite, $\chi$ contains no monochromatic copy of $K_k$. Our task is to show that, with positive probability, every edge receives at least $s$ colours. 

\begin{proof}[Proof of Theorem~\ref{thm:simple:bound}]
Let $e = uv \in E(K_n)$. Since $i \not\in \chi(e)$ only if $\phi_i(u) = \phi_i(v)$, which occurs with probability $1/(k-1)$, independently for each colour, it follows that
\begin{equation}\label{eq:simple:union}
\Pr\big( |\chi(e)| < s \big) \le {r \choose \eps r} \bigg( \frac{1}{k-1} \bigg)^{\eps r} \le \bigg( \frac{e}{\eps(k-1)} \bigg)^{\eps r}.
\end{equation}
Thus, if $n < \big( \eps (k-1)/e \big)^{\eps r/2}$, then the expected number of edges $e \in E(K_n)$ such that $|\chi(e)| < s$ is less than $1$, and hence there exists a choice of the functions $\phi_1,\ldots,\phi_r$ such that every edge receives at least $s$ colours, as required.
\end{proof}

To prove an exponential bound when $\eps(k-1) < e$, we simply replace the union bound in~\eqref{eq:simple:union} by an application of Chernoff's inequality. 

\begin{proof}[Proof of Theorem~\ref{thm:simple:bound2}]
We may assume that $\delta \le 2$, since otherwise the claimed bound (with $c(\delta)$ an absolute constant) follows from Theorem~\ref{thm:simple:bound}. Let $e \in E(K_n)$, and recall that the event $i \not\in \chi(e)$ occurs with probability $1/(k-1)$, independently for each colour $i \in [r]$. Since $\eps (k-1) \ge 1 + \delta$, it follows by Chernoff's inequality that
$$\Pr\big( |\chi(e)| < s \big) = \Pr\Big( \textup{Bin}\big( r,1/(k-1) \big) \ge \eps r \Big) \le e^{-c(\delta) \eps r}$$
for some constant $c(\delta) = \Omega(\delta^2)$. Thus, if $n < e^{c(\delta) \eps r / 2}$, then the expected number of edges $e \in E(K_n)$ such that $|\chi(e)| < s$ is less than $1$, and hence there exists a choice of the functions $\phi_1,\ldots,\phi_r$ such that every edge receives at least $s$ colours, as required.
\end{proof}

We can now easily deduce Corollary~\ref{cor:general}.

\begin{proof}[Proof of Corollary~\ref{cor:general}]
The upper bound holds by~\eqref{CFHMSV:upper}, which was proved in~\cite[Theorem~1.1]{CFHMSV}, and holds for all $r > s \ge 1$ and all $k \ge 3$. When $s \le r - C\log r$ and $k \ge (C / \eps) \log r$, the lower bound follows from Theorem~\ref{thm:setRamsey} (though note that in the case $r - s = \Omega(r)$ it was first proved in~\cite{CFHMSV}), and when $s \ge r - C\log r$ it follows from a simple random colouring (see~\cite[Section~2]{CFHMSV}) that 
$$R_{r,s}(k) \ge 2^{\eps k / 6}.$$
Finally, when $(1+\delta)/\eps + 1 \le k \le (C / \eps) \log r$, the lower bound follows from Theorem~\ref{thm:simple:bound2}. 
\end{proof}

\section{Concluding remarks}

We expect that determining $\log R_{r,s}(k)$ up to a constant factor for all $s = r - o(r)$ will be extremely difficult (though it does not seem unreasonable to hope that doing so in this range might prove to be easier than in the case $s = 1$). On the other hand, the upper and lower bounds in Corollary~\ref{cor:general} differ by a factor of roughly $(\log r)^2$ in the exponent when $s = r - \log r$, and also when $\eps k = \log r$. We expect that the lower bound can be improved by a factor of $\log r$ in these cases. 

\begin{conj}\label{conj:setRamsey}
There exists a constant 
$\delta > 0$ such that 
\begin{equation}\label{eq:setRamsey:conj}
R_{r,s}(k) \ge \exp\bigg( \frac{\delta k (r-s)^2}{r} \bigg)
\end{equation}
for every $r > s \ge 1$ and $k \ge 3 / \eps$, where $\eps = (r-s)/r$.
\end{conj}


It seems plausible that~\eqref{eq:setRamsey:conj} could be proved in the case $\eps k = O(\log k)$ using a modified version of the construction in Section~\ref{simple:sec} (taking Turán graphs with fewer parts, and assigning the set $[r]$ to edges with $|\chi(e)| < s$), since the density of bad edges will be small, and the correlation between them not too large. The obstruction when $s \ge r - \log r$ is potentially more serious, since in this case the density of bad edges will be larger than $\eps$, so one would not expect the conclusion of Lemma~\ref{lemma:many-bad} to hold. It therefore seems unlikely that the construction from Section~\ref{sec:construction} can be used to prove~\eqref{eq:setRamsey:conj} in this range. 

\section{Acknowledgements}
The research that led to this paper started at WoPOCA 2022. We thank the workshop organisers for a productive working environment. This study was also financed in part by the Coordenação de Aperfeiçoamento de Pessoal de Nível Superior, Brasil (CAPES). 

\end{document}